\documentclass[10pt,twoside]{article}
\usepackage[hmarginratio=1:1,a4paper,text={16cm,23cm}]{geometry}
\usepackage[english]{babel}
\usepackage[utf8]{inputenc}
\usepackage[T1]{fontenc}
\usepackage{amsmath,amsfonts,amssymb,amsthm}
\usepackage{tikz}
	\usetikzlibrary{arrows,calc,decorations.text,decorations.pathreplacing}
\usepackage{dsfont,color,graphicx,mathrsfs,enumerate,pdfpages,setspace}
\usepackage{caption}
\usepackage{url}
\usepackage[backref=page]{hyperref}
\hypersetup{allcolors=blue,colorlinks=true,breaklinks=true,bookmarksopen=true}

\usepackage{amsthm}
\newcounter{count}
\newcounter{countrem}
\newtheorem{theorem}[count]{Theorem}

\newtheorem{proposition}[count]{Proposition}

\newtheorem{assumption}[count]{Assumption}

\makeatletter
\renewcommand{\thefigure}{\ifnum \c@section>\z@ \thesection.\fi
 \@arabic\c@figure}
\@addtoreset{figure}{section}
\makeatother

\makeatletter
\renewenvironment{proof}[1][\proofname]{\par
  \pushQED{\qed}%
  \normalfont \topsep6\p@\@plus6\p@\relax
  \trivlist
  \item[\hskip\labelsep
        \bfseries
    #1\@addpunct{\scantokens{:}}]\ignorespaces
}{%
  \popQED\endtrivlist\@endpefalse
}
\makeatother

\title{A Probabilistic Look at Conservative Growth-Fragmentation Equations}

\author{Florian \textsc{Bouguet}}

\date{
    \emph{Inria Nancy -- Grand Est, Team BIGS}\\
    \emph{Institut \'Elie Cartan de Lorraine}\\[2ex]
    October 20, 2016
}

\usepackage{fancyhdr}
		\makeatletter
			\let\runauthor\@author
			\let\runtitle\@title
		\makeatother
        \fancypagestyle{main}{
                \fancyhf{}
                \fancyhead[CE]{\runtitle}
                \fancyhead[CO]{\runauthor}
                \fancyfoot[c]{\thepage}
                
                }
        \fancypagestyle{plain}{
                \fancyhf{}
                \fancyfoot[c]{\thepage}
                
                }

\newenvironment{remark}[1][]{\par\noindent\refstepcounter{countrem}\textbf{Remark}
\ifx\newenvironment#1\newenvironment
	\textbf{\arabic{countrem}.}
\else
	\textbf{\arabic{countrem}} (#1)\textbf{.}
\fi}{\leavevmode\unskip\penalty9999 \hbox{}\nobreak\hfill\quad\hbox{{$\diamondsuit$}}}
\newenvironment{acknowledgements}{\noindent\textbf{Acknowledgements:}}{}

\definecolor{darkred}{rgb}{0.9,0.1,0.1}

\newcommand*{\e}{\text{e}}
\newcommand*{\E}{\mathbb{E}}

\newcommand*{\Leb}{\mathbb{L}}
\newcommand*{\N}{\mathbb{N}}
\renewcommand*{\P}{\mathbb{P}}

\newcommand*{\R}{\mathbb{R}}

\newcommand*{\eps}{\varepsilon}

\newcommand*{\indic}{\mathds{1}}

\begin{document}

\pagestyle{main}
\maketitle
\begin{center}
\begin{minipage}[c]{.8\textwidth}
\textbf{Abstract:} In this note, we consider general growth-fragmentation equations from a probabilistic point of view. Using Foster-Lyapunov techniques, we study the recurrence of the associated Markov process depending on the growth and fragmentation rates. We prove the existence and uniqueness of its stationary distribution, and we are able to derive precise bounds for its tails in the neighborhoods of both $0$ and $+\infty$. This study is systematically compared to the results obtained so far in the literature for this class of integro-differential equations.


\vspace{1cm}


\end{minipage}
\end{center}

\section{Introduction}
\label{sec:Intro}

In this work, we consider the growth and fragmentation of a population of microorganisms (typically, bacteria or cells) through a structured quantity $x$ which rules the division. For instance, one can consider $x$ to be the size of a bacterium. The bacteria grow and, from time to time, split into two daughters. This behavior leads to an integro-differential equation, which can also model numerous phenomena involving fragmentation, like polymerization, network congestions or neurosciences. In the context of a dividing population, we refer to \cite[Chapter~4]{Per07} for background and biological motivations, and to \cite{Mic06,DG10} for motivations in determining the eigenelements of the equation, which correspond to the Malthusian parameter of the population (see \cite{Per07}). Regardless, if we denote by $u(t,x)$ the concentration of individuals of size $x$ at time $t$, such dynamics lead to the following \emph{growth-fragmentation equation}\index{growth-fragmentation}:
\begin{equation}
\partial_tu(t,x)+\partial_x[\tau(x)u(t,x)]+\beta(x)u(t,x)
=2\int_x^\infty\beta(y)\kappa(x,y)u(t,y)dy,
\label{eq:GrowthFrag}
\end{equation}
for $x,t>0,$ where $\tau$ and $\beta$ are the respective growth rate and fragmentation rate of the population, and $\kappa$ is the fragmentation kernel (here we adopt the notation of \cite{CDG12}). Not that because of the factor 2 in the right-hand side of \eqref{eq:GrowthFrag}, the mass of the total particle system increases with time, so that this equation is not conservative.

The evolution of this population, or rather its probabilistic counterpart, has also been widely studied for particular growth and division rates. In the context of network congestions, this is known as the TCP window size process, which received a lot of attention recently (see \cite{LVL08,CMP10,BCGMZ13,ABGKZ14}). Let us provide the probabilistic interpretation of this mechanism. Consider a bacterium of size $X$, which grows at rate $\tau$ and randomly splits at rate $\beta$ following a kernel $\kappa$, as before. We shall denote by $Q(x,dy):=x\kappa(yx,x)dy$ to deal with the relative size of the daughters compared to the mother's, so that
\[\int_0^xf(y)\kappa(y,x)dy=\int_0^1f(xy)Q(x,dy).\]
We shall naturally assume that, for any $x>0$,
\[\int_0^x\kappa(y,x)dy=\int_0^1Q(x,dy)=1.\]
If we dismiss one of the two daughters and carry on the study only with the other one, the growth and fragmentation of the population can also be modeled by a \emph{Piecewise deterministic Markov process} (PDMP) $(X_t)_{t\geq0}$ with c\`{a}dl\`{a}g trajectories a.s. The dynamics of $X$ are ruled by its infinitesimal generator, defined for any function $f$ in its domain $\mathcal D(\mathcal L)$:
\begin{equation}
\mathcal Lf(x):=\tau(x)f'(x)+\beta(x)\int_0^1[f(xy)-f(x)]Q(x,dy).
\label{eq:genGrowthFrag}
\end{equation}
We shall call $X$ a \emph{cell process} (not to be confused with a growth-fragmentation process, see Remark~\ref{rk:nomenclature}). It is a Feller process, and we denote by $(P_t)_{t\geq0}$ its semigroup (for reminders about Feller processes or PDMPs, see \cite{EK86,Dav93}). If we denote by $\mu_t=\mathscr L(X_t)$ the probability law of $X_t$, the Kolmogorov's forward equation $\partial_t(P_tf)=\mathcal L P_tf$ is the weak formulation of
\begin{equation}
\partial_t\mu_t=\mathcal L'\mu_t,
\label{eq:GrowthFragCons2}
\end{equation}
where $\mathcal L'$ is the adjoint operator of $\mathcal L$ in $L^2(\Leb)$ where $\Leb$ stands for the Lebesgue measure. Now, if $\mu_t$ admits a density $u(t,\cdot)$ with respect to $\Leb$, then \eqref{eq:GrowthFragCons2} writes
\begin{equation}
\partial_tu(t,x)=\mathcal L'u(t,x)=-\partial_x[\tau(x)u(t,x)]-\beta(x)u(t,x)
+\int_x^\infty\beta(y)\kappa(x,y)u(t,y)dy.
\label{eq:GrowthFragCons}
\end{equation}
Note that \eqref{eq:GrowthFragCons} is the conservative version of \eqref{eq:GrowthFrag}, since for any $t\geq0$, $\int_0^\infty u(t,x)dx=1$, which comes from the fact that there is only one bacterium at a time.

\begin{remark}[Link with biology]
\label{rk:biology}
Working with the probabilistic version of the problem allows us not to require the absolute continuity of $\mu_t$ nor $Q(x,\cdot)$. This is useful since many biological models set $Q(x,\cdot)=\delta_{1/2}$ (equal mitosis) or $Q(x,\cdot)=\mathscr U([0,1])$ (uniform mitosis). Note that biological models usually assume that $\int_0^1y Q(x,dy)=1/2$, so that the mass of the mother is conserved after the fragmentation, which is automatically satisfied for a symmetric division, but we do not require this hypothesis in our study. We stress that it is possible to study both daughters with a structure of random tree, as in \cite{BDMT11,Ber15,DHKR15}, the latter also drawing a bridge between the stochastic and deterministic viewpoints.
\end{remark}

In the articles \cite{DG10,CDG12,BCG13}, the authors investigate the behavior of the first eigenvalue and eigenfunction of \eqref{eq:GrowthFrag}, with a focus on the dependence on the growth rate $\tau$ and the division rate $\beta$. Although it has been previously done for specific rates (e.g. \cite{GvWW07}), they work in the setting of general functions $\tau$ and $\beta$. The aim of the present paper is to provide a probabilistic counterpart to the aforementioned articles, by studying the Markov process $(X_t)_{t\geq0}$ generated by \eqref{eq:genGrowthFrag}, and to explain the assumptions for the well-posedness of the problem. We provide a probabilistic justification to the links between the growth and fragmentation rates, with the help of the renowned Foster-Lyapunov criterion. We shall also study the tails of distribution of the stationary measure of the process when it exists. We will see that, although the assumptions are similar, there is a difference between the tails of the stationary distribution in the conservative case and in the non-conservative case.

\begin{remark}[Cell processes and growth-fragmentation processes]
\label{rk:nomenclature}
The name cell process comes from the paper \cite{Ber15}, where the author provides a general construction for the so-called growth-fragmentation processes, with the structure of branching processes. This construction allows to study the family of all bacteria (or cells) alive at time $t$. Let us stress that, in \cite{Ber15}, the process is allowed to divide on a dense set of times; the setting of PDMPs does not capture such a phenomenon, but does not require the process $X$ to converge a.s. at infinity. The construction of a growth-fragmentation process is linked to the study of the non-conservative growth-fragmentation equation \eqref{eq:GrowthFrag}, whereas our construction of the cell process enables us to study the behavior of its conservative version \eqref{eq:GrowthFragCons}. In Section~\ref{sec:balance}, we shall see that there is no major differences for the well-posedness of the equation in the conservative and the non-conservative settings. However, the tails of the stationary distribution are rather different in the two frameworks, so the results of Section~\ref{sec:tail} are to be compared to the computations of \cite{BCG13} when $\lambda=0$. The Malthusian parameter $\lambda$ being the exponential growth rate of the mass of the population, it is clear that it is null in the conservative case.
\end{remark}

The rest of this paper is organized as follows: in Section~\ref{sec:balance}, we study the Harris recurrence of $X$ as well as the existence and uniqueness of its stationary distribution $\pi$, and we compare our conditions to those of \cite{CDG12}. In Section~\ref{sec:tail}, we study the moments of $\pi$, we derive precise upper bounds for its tails of distribution in the neighborhoods of both 0 and $+\infty$ and we compare our conditions to those of \cite{BCG13}.

\section{Balance Between Growth and Fragmentation}
\label{sec:balance}

To investigate the assumptions used in \cite{CDG12}, we turn to the study of the Markov process generated by \eqref{eq:genGrowthFrag}. More precisely, we will provide a justification to the balance between $\tau$ and $\beta$ with the help of a Foster-Lyapunov criterion. Note that we shall not require the fragmentation kernel $Q(x,dy)$ to admit a density with respect to the Lebesgue measure $\Leb(dy)$. Moreover, in order to be as general as possible, we do not stick to the biological framework and thus do not assume that $\int_0^1 Q(x,dy)=1/2,$ which will be (technically) replaced by Assumption~\ref{assumption:GrowthFragMoments}.i) below.

We start by stating general assumptions on the growth and fragmentation rates.
\begin{assumption}[Behavior of $\tau$ and $\beta$]
Assume that:
\begin{enumerate}[i)]
	\item The functions $\beta$ and $\tau$ are continuous, and $\tau$ is locally Lipschitz.
	\item For any $x>0$, $\beta(x),\tau(x)>0$.
	\item There exist constants $\gamma_0,\gamma_\infty,\nu_0,\nu_\infty$ and $\beta_0,\beta_\infty,\tau_0,\tau_\infty>0$ such that
\[ \beta(x)\underset{x\to0}{\sim}\beta_0 x^{\gamma_0},\quad \beta(x)\underset{x\to\infty}{\sim}\beta_\infty x^{\gamma_\infty},\quad\tau(x)\underset{x\to0}{\sim}\tau_0 x^{\nu_0},\quad \tau(x)\underset{x\to\infty}{\sim}\tau_\infty x^{\nu_\infty}.\]
\end{enumerate}
\label{assumption:GrowthFragRates}
\end{assumption}
Note that, if $\tau$ and $\beta$ satisfy Assumption~\ref{assumption:GrowthFragRates}, then Assumptions~(2.18) and (2.19) in \cite{CDG12} are fulfilled (by taking $\mu=|\gamma_\infty|$ or $\mu=|\nu_\infty|$, and $r_0=|\nu_0|$ therein).

The following assumption concerns the expected behavior of the fragmentation, and is easy to check in most cases, especially if $Q(x,\cdot)$ does not depend on $x$. For any $a\in\R$, we define the moment of order $a$ of $Q(x,\cdot)$ by
\[M_x(a):=\int_0^1y^aQ(x,dy),\quad M(a):=\sup_{x>0}M_x(a).\]

\begin{assumption}[Moments of $Q$]
Assume that:
\begin{enumerate}[i)]
	\item There exists $a>0$ such that $M(a)<1$.
	\item There exists $b>0$ such that $M(-b)<+\infty$.
	\item For any $x>0,Q(x,\{1\})=0$.
\end{enumerate}
\label{assumption:GrowthFragMoments}\end{assumption}
Note that, in particular, Assumption~\ref{assumption:GrowthFragMoments}.i) and ii) imply that, for any $x>0,Q(x,\{1\})<1$ and $Q(x,\{0\})=0.$ Assumption~\ref{assumption:GrowthFragMoments}.iii) means that there are no \emph{phantom jumps}, i.e. divisions of the bacteria without loss of mass. It is easy to deal with a process with no phantom jumps with the following thinning technique: if $X$ is generated by \eqref{eq:genGrowthFrag} and $Q$ admits the decomposition
\[Q(x,dy)=Q(x,\{1\})\delta_1+(1-Q(x,\{1\}))Q'(x,dy),\]
then notice that \eqref{eq:genGrowthFrag} writes 
\[\mathcal Lf(x)=\tau(x)f'(x)+\beta'(x)\int_0^1[f(xy)-f(x)]Q'(x,dy),\]
with $\beta'(x)=(1-Q(x,\{1\})\beta(x)$ and $Q'(x,\{1\})=0$.

Let us make another assumption, concerning the balance between the growth rate and the fragmentation rate in the neighborhoods of $0$ and $+\infty$, which is fundamental to obtain an interesting Markov process. 

\begin{assumption}[Balance of $\beta$ and $\tau$]
Assume that
	\[\gamma_0>\nu_0-1,\quad \gamma_\infty>\nu_\infty-1.\]
\label{assumption:GrowthFragBalance}\end{assumption}

Let us mention that Assumptions~\ref{assumption:GrowthFragRates}.iii) and \ref{assumption:GrowthFragBalance}, could be replaced by integrability conditions in the neighborhoods of 0 or $+\infty$, see Assumptions~(2.21) and (2.22) in \cite{CDG12}. However, we make those hypotheses for the sake of simplicity, and for easier comparisons of our results to \cite{CDG12,BCG13}.

\begin{remark}[The critical case]
This remark concerns the whole paper, and may be omitted at first reading. Throughout Section~\ref{sec:balance}, we can weaken Assumption~\ref{assumption:GrowthFragBalance} with the following:
\begin{enumerate}[i)]
	\item Either
	\begin{equation}
	\gamma_0>\nu_0-1,\quad\text{or}\quad\gamma_0=\nu_0-1 \text{ and }\frac b{M(-b)-1}<\frac{\beta_0}{\tau_0}.
	\label{eq:BalanceCriticZero}
	\end{equation}
	\item Either
	\begin{equation}
	\gamma_\infty>\nu_\infty-1\quad\text{or}\quad\gamma_\infty=\nu_\infty-1 \text{ and }\frac a{1-M(a)}<\frac{\tau_\infty}{\beta_\infty}.
	\label{eq:BalanceCriticInfty}
	\end{equation}
\end{enumerate}
Indeed, a careful reading of the proof of Theorem~\ref{thm:GrowthFragErgodicity} shows that computations are similar, and the only change lies in the coefficients in \eqref{eq:preuveGrowthFrag1} and \eqref{eq:preuveGrowthFrag2}, which are still negative under \eqref{eq:BalanceCriticZero} and \eqref{eq:BalanceCriticInfty}. This corresponds to the critical case of the growth-fragmentation equations (see for instance \cite{BW16,DE16}).

However, the behavior of the tail of the stationary distribution changes radically in the critical case. As a consequence, Section~\ref{sec:tail} is written in the framework of Assumption~\ref{assumption:GrowthFragBalance} only. Indeed, it is crucial to be able to choose $a$ as large as possible (which is ensured in Assumption~\ref{assumption:BoundV}), so that $\pi$ admits moments of any order. This is not possible under \eqref{eq:BalanceCriticInfty}, since then
\[\lim_{a\to+\infty}\frac a{1-M(a)}=+\infty,\]
so that the Foster-Lyapunov criterion does not apply and we expect the stationary measure to have heavy tails.
\end{remark}

Define $V$ as a smooth, convex function on $(0,\infty)$ such that 
\begin{equation}
V(x)=
\left\{\begin{array}{ll}
x^{-b}&\text{if }x\in(0,1],\\
x^a&\text{if }x\in[2,\infty),\\
\end{array}\right.
\label{eq:DefV}
\end{equation}
where $a$ and $b$ satisfy Assumption~\ref{assumption:GrowthFragMoments}. We can now state the main result of this article.

\begin{theorem}[Behavior of the cell process]
Let $X$ be a PDMP generated by \eqref{eq:genGrowthFrag}. If Assumptions~\ref{assumption:GrowthFragRates}, \ref{assumption:GrowthFragMoments} and \ref{assumption:GrowthFragBalance} are in force, then $X$ is irreducible, Harris recurrent and aperiodic, compact sets are petite for $X$, and the process possesses a unique (up to a multiplicative constant) stationary measure $\pi$.

Moreover, if
\[b\geq\nu_0-1,\quad a\geq-\gamma_\infty,\]
then $X$ is positive Harris recurrent and $\pi$ is a probability measure.

Furthermore, if
\[\nu_0\leq1,\quad \gamma_\infty\geq0,\]
then $X$ is exponentially ergodic in $(1+V)$-norm.
\label{thm:GrowthFragErgodicity}\end{theorem}

\begin{remark}[Link with the conditions of \cite{CDG12}]
We highlight the equivalence of Assumption~\ref{assumption:GrowthFragBalance} and \cite[Eq.~(2.4) and (2.5)]{CDG12}. The condition \cite[Eq.~(2.6)]{CDG12} writes in our context
\[\int_0^u Q(x,dy)\leq\min\left(1,Cu^{\bar\gamma}\right),\]
which is implied by Assumption~\ref{assumption:GrowthFragMoments}.ii) together with the condition $b\geq\nu_0-1$, as soon as $Q(x,\cdot)$ admits a density with respect to $\Leb$. Let us also mention that counterexamples for the existence of the stationary measure are provided in \cite{DG10}, when $\beta$ is constant and $\tau$ is affine.
\end{remark}

Before proving Theorem~\ref{thm:GrowthFragErgodicity}, let us shortly present the Foster-Lyapunov criterion, which is the main tool for our proof (the interested reader may find deeper insights in \cite{MT93} or \cite{MT93b}). The idea is to find a so-called Lyapunov function $V$ controlling the excursions of $X$ out of petite sets. Recall that a set $K\subseteq\R_+$ is petite if there exists a probability distribution $\mathscr A$ over $\R_+$ and some non-trivial positive measure $\nu$ over $\R_+$ such that, for any $x\in K$, $\int_0^\infty\delta_xP_{t}\mathscr A(dt)\geq\nu$. We produce here three criteria, adapted from \cite[Theorems~3.2, 4.2 and 6.1]{MT93b}, which provide stronger and stronger results. Recall that, for some norm-like fonction $V$, we define the $V$-norm of a probability measure $\mu$ by
\[\|\mu\|_V:=\sup_{|f|\leq V}|\mu(f)|=\sup_{|f|\leq V}\left|\int fd\mu\right|.\]

\begin{theorem}[Foster-Lyapunov criterion]
Let $X$ be a Markov process with c\`{a}dl\`{a}g trajectories a.s. Let $V\geq 1$ be a continuous norm-like real-valued function. Assume that compact sets of $(0,+\infty)$ are petite for $X$.
\begin{enumerate}[i)]
	\item If there exist a compact set $K$ and a positive constant $\alpha'$ such that
	\[\mathcal LV\leq \alpha'\indic_K,\]
	then $X$ is Harris recurrent and possesses a unique (up to a multiplicative constant) stationary measure $\pi$.
	\item Moreover, if there exist a function $f\geq1$ and a positive constant $\alpha$ such that
	\[\mathcal LV\leq -\alpha f+\alpha'\indic_K,\]
	then $X$ is positive Harris recurrent, $\pi$ is a probability measure and $\pi(f)<+\infty$.
	\item Moreover, if $f\geq V$, then $X$ is exponentially ergodic and there exist $C,v>0$ such that
	\[\|\mu_t-\pi\|_{1+V}\leq C(1+\mu_0(V))\e^{-vt}.\] 
\end{enumerate}
\label{thm:FosterLyapunov}
\end{theorem}

Note that the exponential rate $v$ provided in Theorem~\ref{thm:FosterLyapunov} is not explicit; if one wants to obtain quantitative speeds of convergence, it is often useful to turn to ad hoc coupling methods (see \cite{BCGMZ13} for instance). Also, note that Assumption~\ref{assumption:GrowthFragMoments} is sufficient but not necessary to derive ergodicity from a Foster-Lyapunov criterion, since we only need the limits in \eqref{eq:preuveGrowthFrag1} and \eqref{eq:preuveGrowthFrag2} to be negative. Namely, we only ask the fragmentation kernel $Q(x,\cdot)$ to be not too close to 0 and 1, uniformly over $x$.

\begin{remark}[Construction of $V$]
If we are able to prove a Foster-Lyapunov criterion with a norm-like function $V$, we want to choose $V$ as explosive as possible (i.e. such that $V(x)$ goes quickly to $+\infty$ when $x\to0$ or $x\to+\infty$) to obtain better bounds for the tail of $\pi$, since $\pi(V)$ is finite: this is the purpose of Section~\ref{sec:tail}. If we define $V$ with \eqref{eq:DefV}, this choice brings us to choose $a$ and $b$ as large as possible in Assumption~\ref{assumption:GrowthFragMoments}. However, the larger $a$ and $b$, the slower the convergence (because of the term $\mu_0(V)$), so there is a balance to find here.

For many particular cell processes, it is possible to build a Lyapunov function of the form $x\mapsto\e^{\theta x}$, so that $\pi$ admits exponential moments up to $\theta$. We shall use a similar function in Section~\ref{sec:tail} to obtain bounds for the tails of the stationary distribution.
\end{remark}

\begin{proof}[Proof of Theorem~\ref{thm:GrowthFragErgodicity}]
We denote by $\varphi_z$ the unique maximal solution of $\partial_ty(t)=\tau(y(t))$ with initial condition $z$, and let $a,b>0$ be as in Assumption~\ref{assumption:GrowthFragMoments}. Firstly, we prove that compact sets are petite for $(X_t)_{t\geq0}$. Let $z_2>z_1>z_0>0$ and $z\in[z_0,z_1]$. Since $\tau>0$ on $[z_0,z_2]$, the function $\varphi_z$ is a diffeomorphism from $[0,\varphi_z^{-1}(z_2)]$ to $[z,z_2]$; let $t=\varphi^{-1}_{z_0}(z_2)$ be the maximum time for the flow to reach $z_2$ from $[z_0,z_1]$. Denote by $X^z$ the process generated by \eqref{eq:genGrowthFrag} such that $\mathscr L(X_0)=\delta_z$, and $T_n^z$ the epoch of its $n^\text{th}$ jump. Let $\mathscr A=\mathscr U([0,t])$. For any $x\in[z_1,z_2]$, we have
\begin{align}
\int_0^\infty\P(X_s^z\leq x)\mathscr A(ds)&\geq\frac1t\int_0^t\P(X_s^z\leq x|T_1^z>\varphi_z^{-1}(z_2))\P(T_1^z>\varphi_z^{-1}(z_2))ds\notag\\
&\geq\frac{\P(T_1^z>\varphi_z^{-1}(z_2))}t\int_0^t\P(\varphi_z(s)\leq x)ds\notag\\
&\geq\frac{\P(T_1^z>\varphi_z^{-1}(z_2))}t\int_0^{\varphi_z^{-1}(x)}ds\notag\\
&\geq\frac{\P(T_1^z>\varphi_z^{-1}(z_2))}t\int_z^x(\varphi_z^{-1})'(u)du.\label{eq:preuveGrowthFrag3}
\end{align}
Since $\beta$ and $\tau$ are bounded on $[z_0,z_2]$, the following inequalities hold:
\begin{align*}
\P(T_1^z>\varphi_z^{-1}(z_2))&=\exp\left(-\int_0^{\varphi_z^{-1}(z_2)}\beta(\varphi_z(s))ds\right)\\
&=\exp\left(-\int_z^{z_2}\beta(u)(\varphi_z^{-1})'(u)du\right)\\
&\geq\exp\left(-(z_2-z_0)\sup_{[z_0,z_2]}\left(\beta(\varphi_z^{-1})'\right)\right)\\
&\geq\exp\left(-(z_2-z_0)\left(\sup_{[z_0,z_2]}\beta\right)\left(\inf_{[z_0,z_2]}\tau\right)^{-1}\right),
\end{align*}
since $\sup_{[z_0,z_2]}(\varphi_z^{-1})'=\left(\inf_{[z_0,z_2]}\tau\right)^{-1}$. Hence, there exists a constant $C$ such that, \eqref{eq:preuveGrowthFrag3} writes, for $x\in[z_1,z_2]$,
\[\int_0^\infty\P(X_s^z\leq x)\mathscr A(ds)\geq C(x-z_1),\]
which is also
\[\int_0^\infty\delta_zP_s\mathscr A(ds)\geq C\Leb_{[z_1,z_2]},\]
where $\Leb_K$ is the Lebesgue measure restricted to a Borelian set $K$. Hence, by definition, $[z_0,z_1]$ is a petite set for the process $X$.

Now, let us show that the process $(X_t)$ is $\Leb_{(0,\infty)}$-irreducible with similar arguments. Let $z_1>z_0>0$ and $z>0$. If $z\leq z_0$,
\begin{align}
\E\left[\int_0^\infty\indic_{\{z_0\leq X_t^z\leq z_1\}}dt\right]&\geq\P(T_1^z>\varphi_z^{-1}(z_1))\E\left[\left.\int_0^\infty\indic_{\{z_0\leq X_t^z\leq z_1\}}dt\right|T_1^z>\varphi^{-1}_z(z_1)\right]\notag\\
&\geq \exp\left(-(z_1-z_0)\left(\sup_{[z_0,z_1]}\beta\right)\left(\inf_{[z_0,z_1]}\tau\right)^{-1}\right)\varphi^{-1}_{z_0}(z_1).\label{eq:preuveGrowthFrag4}
\end{align}
If $z>z_0$, for any $t_0>0$ and $n\in\N$, the process $X^z$ has a positive probability of jumping $n$ times before time $t_0$. Recall that $\int_0^1y^aQ(x,dy)\leq M(a)<1$. For any $n>(\log(z)-\log(z_0))\log(M(a)^{-1})^{-1}$, let $0<\varepsilon<z_0^a-(zM(a)^n)^a$. By continuity of $(x,t)\mapsto\varphi_x(t)$, there exists $t_0>0$ small enough such that, $\forall (x,t)\in[0,z]\times[0,t_0]$,
\[\varphi_x(t)^a\leq x^a+\frac\eps{n+1},\quad\E[(X_{t_0}^z)^a|T_n^z\leq t_0]\leq (zM(a)^n)^a+\varepsilon<z_0^a.\]
Then, using Markov's inequality
\[\P(X_{t_0}^z\leq z_0|T_n^z\leq t_0<T_{n+1}^z)\geq1-\frac{\E[(X_{t_0}^z)^a|T_n^z\leq t_0<T_{n+1}^z]}{z_0^a}>0.\]
Then, $\P(X_{t_0}^z\leq z_0)>0$ for any $t_0$ small enough, and, using \eqref{eq:preuveGrowthFrag4}
\begin{align*}
&\E\left[\int_0^\infty\indic_{\{z_0\leq X_t^z\leq z_1\}}dt\right]\geq\E\left[\left.\int_{t_0}^\infty\indic_{\{z_0\leq X_t^z\leq z_1\}}dt\right|X_{t_0}^z\leq z_0\right]\P(X_{t_0}^z\leq z_0)\\
&\quad\geq \exp\left(-(z_1-z_0)\left(\sup_{[z_0,z_1]}\beta\right)\left(\inf_{[z_0,z_1]}\tau\right)^{-1}\right)\varphi^{-1}_{z_0}(z_1)\P(X_{t_0}^z\leq z_0)\\
&\quad>0.
\end{align*}
Aperiodicity is easily proven with similar arguments.

We turn to the proof of the Lyapunov condition. For $x\geq 2$, $V(x)=x^a$ and
\begin{align}
\mathcal LV(x) &=a\frac{\tau(x)}xV(x)+\beta(x)\int_0^1V(xy)Q(x,dy)-\beta(x)V(x)\notag\\
&\leq\left(a\frac{\tau(x)}x-\beta(x)\right)V(x)+\beta(x)\int_0^{1/x}(xy)^{-b}Q(x,dy)\notag\\
&\quad+\beta(x)\int_{1/x}^{2/x}2^aQ(x,dy) +\beta(x)\int_{2/x}^1(xy)^aQ(x,dy)\notag\\
&\leq\left(a\frac{\tau(x)}x-\beta(x)\right)V(x)+\beta(x)\left(x^{-b}M_x(-b)+2^a+x^aM_x(a)\right)\notag\\
&\leq\left(a\frac{\tau(x)}x-\beta(x)\left(1-M_x(a)-\frac{M_x(-b)}{x^bV(x)}-\frac{2^a}{V(x)}\right)\right)V(x).\label{eq:preuveGrowthFrag5}
\end{align}For $x\leq 1$, $V(x)=x^{-b}$ and
\begin{equation}
\mathcal LV(x)=\left(-b\frac{\tau(x)}x+\beta(x)(M_x(-b)-1)\right)V(x).
\label{eq:preuveGrowthFrag6}
\end{equation}
Combining $\gamma_\infty>\nu_\infty-1$ with Assumption~\ref{assumption:GrowthFragMoments}.i), for $x$ large enough we have
\begin{align*}
&a\frac{\tau(x)}x-\beta(x)\left(1-M_x(a)-\frac{M_x(-b)}{x^bV(x)}-\frac{2^a}{xV(x)}\right)\\
&\quad\leq a\frac{\tau(x)}x-\beta(x)\left(1-M(a)+o(1)\right)\leq0.
\end{align*}
Likewise, combining $\gamma_0>\nu_0-1$ with Assumption~\ref{assumption:GrowthFragMoments}.ii),
\[-b\frac{\tau(x)}x+\beta(x)(M_x(-b)-1)\leq-b\frac{\tau(x)}x+\beta(x)\left(M(-b)-1\right)\leq0\]
for $x$ close enough to 0. Then, Theorem~\ref{thm:FosterLyapunov}.i) entails that $X$ is Harris recurrent, thus admits a unique stationary measure (see for instance \cite{KM94}).

Note that \eqref{eq:preuveGrowthFrag5} writes
\[\mathcal LV(x)\leq-\beta_\infty(1-M(a)+o(1))x^{a+\gamma_\infty},\]
so that, if we can choose $a\geq-\gamma_\infty$, then
\[\lim_{x\to\infty}-\beta_\infty(1-M(a)+o(1))x^{a+\gamma_\infty}<0.\]
Likewise, \eqref{eq:preuveGrowthFrag6} writes
\begin{equation}
\mathcal LV(x)\leq-(b\tau_0+o(1))x^{\nu_0-1-b},
\label{eq:preuveGrowthFrag8}
\end{equation}
so, if $b\geq\nu_0-1$, we get
\[\lim_{x\to0}-(b\tau_0+o(1))x^{\nu_0-1-b}<0.\]
Then, there exist positive constants $A,\alpha,\alpha'$
\[\mathcal LV\leq-\alpha f+\alpha'\indic_{[1/A,A]},\]
where $f\geq1$ is a smooth function, such that $f(x)=x^{\nu_0-1-b}$ for $x$ close to 0, and $f(x)=x^{a+\gamma_\infty}$ for $x$ large enough. Then, Theorem~\ref{thm:FosterLyapunov}.ii) ensures positive Harris recurrence for $X$.

Now, if we assume $\gamma_\infty\geq0$ and $\nu_0\leq1$ in addition, then there exists $\alpha>0$ such that
\begin{align}
&\lim_{x\to+\infty}a\frac{\tau(x)}x-\beta(x)\left(1-M_x(a)-\frac{M_x(-b)}{x^bV(x)}-\frac{2^a}{xV(x)}\right)\notag\\
&\quad\leq\lim_{x\to+\infty}a\frac{\tau(x)}x-\beta(x)\left(1-M(a)+o(1)\right)\leq-\alpha,
\label{eq:preuveGrowthFrag1}\end{align}
and
\begin{equation}
\lim_{x\to0}-b\frac{\tau(x)}x+\beta(x)(M_x(-b)-1)\leq\lim_{x\to0}-b\frac{\tau(x)}x+\beta(x)\left(M(-b)-1\right)\leq-\alpha.
\label{eq:preuveGrowthFrag2}\end{equation}
Combining \eqref{eq:preuveGrowthFrag5} and \eqref{eq:preuveGrowthFrag6} with \eqref{eq:preuveGrowthFrag1} and \eqref{eq:preuveGrowthFrag2} respectively, and since $V$ is bounded on $[1,2]$, there exist positive constants $A,\alpha'$ such that
\[\mathcal LV\leq-\alpha V+\alpha'\indic_{[1/A,A]}.\]
The function $V$ is thus a Lyapunov function, for which Theorem~\ref{thm:FosterLyapunov}.iii) entails exponential ergodicity for $X$.
\end{proof}

\section{Tails of the Stationary Distribution}
\label{sec:tail}

In this section, we use, and reinforce when necessary, the results of Theorem~\ref{thm:GrowthFragErgodicity} to study the asymptotic behavior of the tails of distribution of the stationary measure $\pi$. We will naturally divide this section into two parts, to study the behavior of $\pi(dx)$ as $x\to0$ and as $x\to+\infty$. Hence, throughout this section, we shall assume that $X$ satisfies Assumptions~\ref{assumption:GrowthFragRates}, \ref{assumption:GrowthFragMoments} and \ref{assumption:GrowthFragBalance}. The key point is to use the fact that $\pi(f)<+\infty$ provided in the second part of Theorem~\ref{thm:FosterLyapunov}. We recall that $\Leb$ stands for the Lebesgue measure on $\R$.

In order to compare our results to those of \cite[Theorem~1.8]{BCG13}, we consider the same framework and make the following assumption:

\begin{assumption}[Density of $Q$ and $\pi$]
Assume that:
\begin{enumerate}[i)]
	\item For any $x>0$, $Q(x,\cdot)\ll\Leb$ and $Q(x,dy)=q(y)dy$, and there exist constants $q_0,q_1\geq0$ and $\mu_0,\mu_1>-1$ such that
	\[q(x)\underset{x\to0}{=}q_0x^{\mu_0}+o(x^{\mu_0}),\quad q(x)\underset{x\to1}{=}q_1(1-x)^{\mu_1}+o((1-x)^{\mu_1}).\]
	\item $\pi\ll\Leb$ and $\pi(dx)=G(x)dx$, and there exist constants $G_0,G_\infty,\widetilde G_\infty>0$ and $\alpha_0,\alpha_\infty,\widetilde \alpha_\infty\in\R$ such that
	\[G(x)\underset{x\to0}{\sim}G_0x^{\alpha_0},\quad G(x)\underset{x\to+\infty}{\sim}G_\infty x^{\alpha_\infty}\exp\left(-\widetilde G_\infty x^{\widetilde\alpha_\infty}\right).\]
\end{enumerate}
\label{assumption:GrowthFragDensity}\end{assumption}

We do not require the coefficients $q_0,q_1$ to be (strictly) positive, so that this assumption can also cover the case $Q(x,dy)=\delta_{r}(dy)$ for $0<r<1$, which is widely used for modeling physical or biological situations. For the sake of simplicity, the hypotheses concerning the density of $\pi$ (resp. $Q$) in the neighborhood of both $0$ and $+\infty$ (resp. $0$ and $1$) are gathered in Assumption~\ref{assumption:GrowthFragDensity}, but it is clear that we only need either the assumption on the left behavior or on the right behavior to precise the fractional moments or the exponential moments of the stationary distribution. In the same spirit, we could weaken $Q(x,\cdot)\ll\Leb$ into $Q(x,\cdot)$ admitting a density with respect to $\Leb$ only in the neighborhoods of 0 and 1, bounded above by $q$.

\begin{remark}[Absolute continuity of $\pi$]
At first glance, Assumption~\ref{assumption:GrowthFragDensity}.ii) may seem disconcerting since $\pi$ is unknown; this is the very goal of this section to study its moments. However, for some models it is possible to prove the absolute continuity of $\pi$, or even get a non-tractable formula for its density (see e.g. \cite{DGR02,GK09,BCGMZ13} for the particular case of the TCP window size process). In such cases, the question of existence of its moments is still not trivial. Still, Assumption~\ref{assumption:GrowthFragDensity}.ii) is stated only to make easier comparisons with the estimates obtained with deterministic methods, and is not needed for the important results of the present paper. However, we stress that Assumption~\ref{assumption:GrowthFragDensity}.i) is in a way more fundamental, since it implies directly Assumption~\ref{assumption:BoundV}, which is needed to study the behavior of $\pi$ in the neighborhood of $+\infty$ (see Proposition~\ref{prop:assumption8}).
\end{remark}

\begin{theorem}[Negative moments of $\pi$]
Let $X$ be the PDMP generated by \eqref{eq:genGrowthFrag}. If Assumptions~\ref{assumption:GrowthFragRates},  \ref{assumption:GrowthFragMoments} and \ref{assumption:GrowthFragBalance} hold, and if
\begin{equation}
b\geq\nu_0-1,
\label{eq:balanceGammaNuMuZero}
\end{equation}
then
\[\int_0^1x^{\nu_0-1-b}\pi(dx)<+\infty.\]

Moreover, if Assumption~\ref{assumption:GrowthFragDensity} holds and $\mu_0+2-\nu_0>0$, then
\[\alpha_0\geq\mu_0+1-\nu_0.\]
\label{thm:tailZero}
\end{theorem}

\begin{proof}
The first part of the theorem is a straightforward consequence of \eqref{eq:preuveGrowthFrag8}.

Combining Assumption~\ref{assumption:GrowthFragMoments}.ii) with Assumption~\ref{assumption:GrowthFragDensity}.i), we naturally have to take $b<\mu_0+1$. Thus, for any $\eps\in(0,\mu_0+1)$, we take $b=\mu_0+1-\eps$. Define $V$ with \eqref{eq:DefV} as before, so that, for $x\leq 1$, $V(x)=x^{-\mu_0-1+\eps}$ and
\[\mathcal LV(x)\leq\left(-b\frac{\tau(x)}x+\beta(x)(M(-b)-1)\right)V(x)\underset{x\to0}{\sim}-b\tau_0x^{\nu_0-\mu_0-2+\eps}.\]
Applying Theorem~\ref{thm:FosterLyapunov}.ii) with $f(x)=x^{\nu_0-\mu_0-2+\eps}$, which tends to $+\infty$ when $\mu_0+1-\nu_0>-1$, we have $\pi(f)<+\infty$ so
\[\int_0^1x^{\alpha_0+\nu_0-\mu_0-2+\eps}dx<+\infty,\quad \alpha_0>1+\mu_0-\nu_0+\eps,\]
for any $\eps>0$. Thus $\alpha_0\geq\mu_0+1-\nu_0$.
\end{proof}

Now, we turn to the study of the tail of distribution of $\pi(dx)$ as $x\to+\infty$. Since choosing a polynomial function as a Lyapunov function can only provide the existence of moments for $\pi$, we need to introduce a more coercive function to study in detail the behavior of its tail of distribution and get the existence of exponential moments. We begin with the following assumption.

\begin{assumption}[Uniform asymptotic bound of the fragmentation]
Let $\theta=\gamma_\infty+1-\nu_\infty$. Assume there exists $0<C<1$ such that, for any $\eps>0$ and $0<\eta<\beta_\infty(\theta\tau_\infty)^{-1}$, there exists $x_0>0$ such that
\[\sup_{x\geq x_0}\int_0^1y^{-\eps}\exp\left({\eta x^\theta(y^\theta-1)}\right)Q(x,dy)<1-C.\]
\label{assumption:BoundV}
\end{assumption}

It is easy to understand this assumption if
\begin{equation}
\widetilde V(x)=x^{-\eps}\e^{\eta x^\theta}
\label{eq:DefV2}
\end{equation}
and if $\mathscr L(Y^{(x)})=Q(x,.)$; then, Assumption~\ref{assumption:BoundV} rewrites
\[\sup_{x\geq x_0}\frac{\E[\widetilde V(xY^{(x)})]}{\widetilde V(x)}\leq 1-C.\]
Once again, this is asking the fragmentation kernel to be not too close to 1. As we will see, this is quite natural when $Q$ has a regular behavior around $0$ and $1$.

\begin{proposition}
\label{prop:assumption8}
Assumption~\ref{assumption:BoundV} holds for any $C\in(0,1)$ whenever Assumption~\ref{assumption:GrowthFragDensity}.i) holds.
\end{proposition}

\begin{proof}
Define $\widetilde V$ as in \eqref{eq:DefV2} for $x\geq 1$, larger than 1, and increasing and smooth on $\R$. For any (large) $x>0$, for any (small) $\delta>0,$
\begin{align}
\E[\widetilde V(xY^{(x)})]&=\E[\widetilde V(xY^{(x)})|Y^{(x)}\leq1-\delta]\P(Y^{(x)}\leq1-\delta) +\E\left[\widetilde V(xY^{(x)})\indic_{\{Y^{(x)}>1-\delta\}}\right]\notag\\
&\leq \widetilde V((1-\delta) x)+\E\left[\widetilde V(xY^{(x)})\indic_{\{Y^{(x)}>1-\delta\}}\right].
\label{eq:proofLemmaBoundV1}
\end{align}
It is clear that, for $\delta <1$,
\begin{equation*}
\lim_{x\to+\infty}\frac{\widetilde V((1-\delta)x)}{\widetilde V(x)}=\lim_{x\to+\infty}(1-\delta)^{-\eps}\exp\left(-\eta(1-(1-\delta)^\theta)x^\theta\right)=0.
\end{equation*}
On the other hand, using Hölder's inequality with $q>\max(1,-1/\mu_1)$ and $p^{-1}+q^{-1}=1$, as well as a Taylor expansion, there exists some constant $C_\delta\geq 1$ such that 
\begin{align*}
&\int_{1-\delta}^1\widetilde V(xy)q_1(1-y)^{\mu_1}dy= q_1\widetilde V(x)\int_0^\delta \exp\left(\eta x^\theta((1-y)^\theta-1)\right)y^{\mu_1}(1-y)^{-\eps}dy\\
&\quad\leq \frac{q_1}{(1-\delta)^\epsilon} \left[\int_0^\delta y^{q\mu_1}dy\right]^{1/q}\widetilde V(x) \left[\int_0^\delta\exp\left(\eta p x^\theta((1-y)^\theta-1)\right)dy\right]^{1/p}\\
&\quad\leq C_\delta \widetilde V(x) \left[\int_0^\delta\exp\left(-\eta p \theta x^\theta y\right)dy\right]^{1/p}\leq C_\delta \widetilde V(x) \left[\frac{1-\exp\left(-\eta p \theta x^\theta \delta\right)}{\eta p \theta x^\theta}\right]^{1/p}.
\end{align*}
The term $(\eta p \theta x^\theta)^{-1}(1-\exp\left(-\eta p \theta x^\theta \delta\right))$ converges to 0 as $x\to+\infty$, so that, for any $C\in(0,1)$, there exists $x_0>0$ such that, for any $x\geq x_0$,
\[\left[\frac{1-\exp\left(-\eta p \theta x^\theta \delta\right)}{\eta p \theta x^\theta}\right]^{1/p}\leq\frac{1-C}{2C_\delta},\quad \frac{\widetilde V((1-\delta)x)}{\widetilde V(x)}\leq \frac {1-C}2.\]
Plugging these bounds into \eqref{eq:proofLemmaBoundV1} achieves the proof.
\end{proof}

Now, we can characterize the weight of the asymptotic tail of $\pi$ and recover \cite[Theorem~1.7]{BCG13}.

\begin{theorem}[Exponential moments of $\pi$]
Let $X$ be the PDMP generated by \eqref{eq:genGrowthFrag}. If Assumptions~\ref{assumption:GrowthFragRates},  \ref{assumption:GrowthFragMoments}, \ref{assumption:GrowthFragBalance} and \ref{assumption:BoundV} hold, then
\[\int_1^{+\infty}x^{\nu_\infty-1-\eps}\exp\left(\eta x^\theta\right)\pi(dx)<+\infty, \quad\theta=\gamma_\infty+1-\nu_\infty,\quad\eta=\frac{C\beta_\infty}{\theta\tau_\infty},\quad\eps>0.\]

Moreover, if Assumption~\ref{assumption:GrowthFragDensity} is also in force, then either:
\begin{itemize}
	\item $\widetilde\alpha_\infty>\gamma_\infty+1-\nu_\infty$;
	\item $\widetilde\alpha_\infty=\gamma_\infty+1-\nu_\infty$ and $\widetilde G_\infty>C\beta_\infty((\gamma_\infty+1-\nu_\infty)\tau_\infty)^{-1}$;
	\item $\widetilde\alpha_\infty=\gamma_\infty+1-\nu_\infty,\widetilde G_\infty=C\beta_\infty((\gamma_\infty+1-\nu_\infty)\tau_\infty)^{-1}$ and $\alpha_\infty\geq-\nu_\infty$.
\end{itemize}
\label{thm:tailInfinity}
\end{theorem}

\begin{remark}[Link with the estimates of \cite{BCG13}]
Note that the hypotheses~\eqref{assumption:GrowthFragBalance} and \eqref{eq:balanceGammaNuMuZero} corresponds to the assumptions required for \cite[Theorem~1.8]{BCG13} to hold, with the correspondence
\[\mu_0\leftrightarrow\mu-1,\quad\nu_0\leftrightarrow\alpha_0,\quad \mu_0+2-\nu_0>0\leftrightarrow \mu+1-\alpha_0>0.\]
Actually, the authors also assume this strict inequality  to prove the existence of the stationary distribution, which we relax here, and we need it only in Theorem~\ref{eq:balanceGammaNuMuZero} to provide a lower bound for $\alpha_0$, which rules the tail of the stationary distribution in the neighborhood of 0. By using Lyapunov methods, there is no hope in providing an upper bound for $\alpha_0$, but we can see that this inequality is optimal by comparing it to \cite[Theorem~1.8]{BCG13} so that, in fact,
\[\alpha_0=\mu_0+1-\nu_0.\]
If $\nu_0>1$, we do not recover the same equivalents for the distribution of $G$ around 0. This is linked to the fact that there is a phase transition in the non-conservative equation at $\nu_0=1$, since the tail of $G$ relies deeply on the function
\[\Lambda(x)=\int_1^x\frac{\lambda+\beta(y)}{\tau(y)}dy,\]
where $\lambda$ is the Malthusian parameter of the equation, which is the growth of the profiles of the integro-differential equation. However, we deal here with the conservative case, for which this parameter is null. The bounds that we provide are indeed consistent with the computations of the proof of \cite[Theorem~1.8]{BCG13} in the case $\lambda=0$.

Concerning the estimates as $x\to+\infty$, as mentioned above, we can not recover upper bounds, and then sharp estimates, for $\alpha_\infty,G_\infty,\widetilde\alpha_\infty$ with Foster-Lyapunov methods. From the proof of Theorem~\ref{thm:tailInfinity}, it is clear that the parameters $\eta$ and $\theta$ are optimal if one wants to apply Theorem~\ref{thm:FosterLyapunov}. Under second-order-type assumptions like \cite[Hypothesis~1.5]{BCG13}, it is clear that
\[\Lambda(x)=\int_1^x\frac{\beta(y)}{\tau(y)}dy\underset{x\to+\infty}{\sim}\frac{\beta_\infty}{\tau_\infty(\gamma_\infty+1-\nu_\infty)}x^{\gamma_\infty+1-\nu_\infty}.\]
This explains the precise value of $\eta$, but we pay the price of having slightly less general hypotheses about $Q$ than \cite{BCG13} with a factor $C$ arising from Assumption~\ref{assumption:BoundV}, which leads to have no disjunction of cases for $\alpha_\infty$. Also, since we deal with the case $\lambda=0$, the equivalent of the function $\Lambda$ is different from the aforementioned paper when $\gamma_\infty<0$, so that $\max\{\gamma_\infty,0\}$ does not appear in our computations.
\end{remark}

\begin{proof}[Proof of Theorem~\ref{thm:tailInfinity}]
Let $\widetilde V$ be as in \eqref{eq:DefV2}, that is
\[\widetilde V(x)=x^{-\eps}\e^{\eta x^\theta},\]
with $\eta,\theta$ given in Theorem~\ref{thm:tailInfinity}. Then, following the computations of the proof of Theorem~\ref{thm:GrowthFragErgodicity}, we get, for $x>x_0$,
\begin{align*}
\mathcal L\widetilde V(x)&\leq\left(\eta\theta \tau_\infty x^{\theta-1+\nu_\infty}-C\beta_\infty x^{\gamma_\infty}-\eps\tau_\infty x^{\nu_\infty-1}\right)(1+o(1))\widetilde V(x)\\
&\leq -\eps\tau_\infty x^{\nu_\infty-1}(1+o(1))\widetilde V(x)\notag\\
&\leq-\frac{\eps\tau_\infty}2x^{\nu_\infty-1}\widetilde V(x).
\end{align*}
Using Theorem~\ref{thm:FosterLyapunov}.ii) with $f(x)=x^{\nu_\infty-1}\widetilde V(x)$, the last inequality ensures that
\[\int_1^{+\infty}f(x)\pi(dx)<+\infty.\]

Now, in the setting of Assumption~\ref{assumption:GrowthFragDensity}, the following holds:
\begin{equation}
\int_1^{+\infty}f(x)\pi(dx)<+\infty\Longleftrightarrow\int_1^{+\infty}x^{\nu_\infty-1-\eps+\alpha_\infty}\exp\left(\eta x^\theta-\widetilde G_\infty x^{\widetilde\alpha_\infty}\right)dx<+\infty.
\label{eq:proofTailInfty1}
\end{equation}
It is clear then that the disjunction of cases of Theorem~\ref{thm:tailInfinity} is the only way for the integral on the right-hand side of \eqref{eq:proofTailInfty1} to be finite.
\end{proof}

\begin{acknowledgements}
The author wants to thank Pierre Gabriel for fruitful discussions about growth-fragmentation equations, as well as Eva L\"{o}cherbach,  Florent Malrieu and Jean-Christophe Breton for their precious help and comments. The referee is also warmly thanked for his constructive remarks. This work was financially supported by the ANR PIECE (ANR-12-JS01-0006-01), and the Centre Henri Lebesgue (programme "Investissements d'avenir" ANR-11-LABX-0020-01).
\end{acknowledgements}

\bibliography{Biblio}

\end{document}